\documentclass[10pt]{amsart}

\usepackage[all]{xy}

\usepackage{amsmath, amssymb, amsthm, latexsym, amscd}

\numberwithin{equation}{section}

\newtheorem{theorem}[subsection]{Theorem}
\newtheorem{prop}[subsection]{Proposition}

\theoremstyle{definition}

\newtheorem{remark}[subsection]{Remark}

\title{Elliptic genera of Berglund-H\"ubsch models }

\author{Minxian Zhu}
\address{Department of Mathematics, 110 Frenlinghuysen Rd, Rutgers University, Piscataway, NJ 08854}
\email{minxian@math.rutgers.edu}

\begin{document} 
\maketitle

\begin{abstract}
We match the elliptic genus of 
a Berglund-H\"ubsch model
with the supertrace of $y^{J[0]}q^{L[0]}$ 
on a vertex algebra $V_{{\bf 1}, {\bf 1}}$. 
We show that it
is a weak Jacobi form and 
the elliptic genus of 
one theory is equal to (up to a sign)
the elliptic genus of its mirror. 
\end{abstract}

\section{Introduction}

Classical mirror symmetry observes the interchange of $h^{1, 1}$ and $h^{1, 2}$
for a pair of Calabi-Yau three-folds. 
Among the millions of examples constructed by the physicists, 
there is one simple yet elegant construction 
by Berglund and H\"ubsch ([BH]).


Berglund and H\"ubsch considered 
a non-degenerate polynomial potential $W$
which defines a hypersurface $X_W$
in a weighted projective space. 
The polynomial $W$ is assumed to 
have as many monomials as the number of variables. 
Transposing the exponents matrix
one obtains the dual potential $W^\vee$
which defines a hypersurface $X_{W^\vee}$
in another weighted projective space. 
Berglund and H\"ubsch described 
the mirror of $X_W$ to be 
the quotient $X_{W^\vee}/H$ 
for some finite abelian group $H$. 
In general, 
$X_W/G$ and $X_{W^\vee}/G^\vee$
are expected to form a mirror pair 
for a subgroup $G \subset \text{Aut}(W)$ 
and its dual group $G^\vee \subset \text{Aut}(W^\vee)$
defined in the appropriate sense. 
It was recently proved in [CR] that 
$X_{W^\vee}/G^\vee$ 
is indeed the mirror of $X_W/G$ 
in the classical sense. 

The elliptic genera of Berglund-H\"ubsch models 
were computed in [BHe]
where it was observed to satisfy the expected duality. 
In this paper,  
we identify the elliptic genus of $X_W/G$
with the supertrace of 
some operator on a vertex algebra, 
and then prove that the elliptic genera of mirror models are equal. 
Our approach relies on the vertex algebra approach 
to mirror symmetry developed by L. Borisov [B1]. 

One of the best understood settings of mirror symmetry is 
Batyrev's construction of Calabi-Yau hypersurfaces in toric varieties [B]
and later generalized by Batyrev and Borisov 
to complete intersections in Gorenstein toric varieties [BB]. 
In the toric case, 
the mirror symmetry is interpreted as polar duality. 
In [B1], 
a vertex algebra $V_{f, g}$ was constructed
from the combinatorial data of dual polytopes $\Delta, \Delta^\vee$ and 
generic choices of coefficients $f, g$
for the lattice points in $\Delta, \Delta^\vee$. 
In particular, 
the elliptic genus of a Calabi-Yau hypersurface in a toric variety 
can be formulated as the supertrace of
some operator on the vertex algebra $V_{f, g}$, 
which made it possible
to prove that the elliptic genera of 
a Calabi-Yau hypersurface in a toric variety and its mirror 
coincide up to a sign ([BL]). 

Recently, 
a vertex algebra approach to Berglund-H\"ubsch mirror symmetry
was introduced in [B2]. 
Similar to the toric case, 
a vertex algebra $V_{{\bf 1}, {\bf 1}}$ was constructed from
a non-degenerate polynomial potential $W$ 
which contains the A and B rings of the theory as subspaces. 
Borisov was then able to prove that the A ring of 
a Berglund-H\"ubsch potential
is isomorphic to the B ring of the dual potential
with the appropriate choices of orbifoldizations. 
In Section 4, 
we identify the elliptic genus of $X_W/G$
with the supertrace of certain operator on the
vertex algebra $V_{{\bf 1}, {\bf 1}}$. 

The paper is organized as follows:
Section 2 recalls the combinatorial data of the 
Berglund-H\"ubsch construction. 
The vertex algebra $V_{{\bf 1}, {\bf 1}}$ 
is defined in Section 3. 
The comparison of the elliptic genus of $W/G$ with the double-graded
superdimension of $V_{{\bf 1}, {\bf 1}}$, 
and the comparison of elliptic genera of a mirror pair 
are done in Section 4.


I am deeply indebted to Lev Borisov for introducing me to 
the vertex algebra approach to mirror symmetry, 
and answering my endless questions.

\section{The Berglund-H\"ubsch mirror symmetry construction}

We will use notations from both [B2] and [CR].
Consider a non-degenerate polynomial potential 
\begin{eqnarray}
W(x_1, \cdots, x_d) & = & 
\sum_{i=1}^d \prod_{j=1}^d x_j^{a_{ij}}
\end{eqnarray}
with invertible exponent matrix $A = (a_{ij})$. 
The variables $x_j$ can be assigned positive
rational degrees $q_j$ 
which makes $W$ homogeneous of degree $1$, 
i.e. 
\begin{eqnarray}
\sum_j a_{ij} q_j & = & 1
\end{eqnarray}
for all $i$. 
Non-degeneracy means that 
the hypersurface $W = 0$ in $\mathbb C^d$ 
is smooth away from the origin. 
This is a very restrictive condition; 
in fact Kreuzer and Skarke classified all
non-degenerate potentials [KS]. 
They are sums of decoupled invertible potentials of 
the following types
\begin{eqnarray}
W_{\text{Fermat}} & = & x^a \\
W_{\text{loop}} & = & 
x_1^{a_1} x_2 + x_2^{a_2} x_3 + \cdots + x_{n-1}^{a_{n-1}}x_n + x_n^{a_n} x_1 \\
W_{\text{chain}} & = & 
x_1^{a_1} x_2 + x_2^{a_2} x_3 + \cdots + x_{n-1}^{a_{n-1}}x_n + x_n^{a_n}.
\end{eqnarray}
Decoupled means that 
the set of variables 
$\{x_1, \cdots, x_d\}$
is partitioned into a disjoint union of subsets,
and the variables in each subset
contribute a polynomial of one of the above types. 

Consider the group $\text{Aut} (W)$ 
of diagonal automorphisms
\begin{eqnarray}
\gamma:  x_j  & \mapsto & \gamma_j x_j
\end{eqnarray}
that preserve the potential $W$, 
that is 
\begin{eqnarray}
\text{Aut} (W) & = & 
\{ \gamma = (\gamma_j): \prod_j \gamma_j^{a_{ij}} =1 \text{  for all } i\}
\end{eqnarray}
Since the matrix $A$ is invertible, 
each $\gamma_j$ is a root of unity. 
If we write 
\begin{eqnarray}
\gamma_j & = &  \text{exp} (2 \pi i p_j)
\end{eqnarray}
for some rational number $p_j$ 
(determined up to an integer), 
then the defining relation of $\text{Aut} (W)$
translates to 
\begin{eqnarray}
\sum_{j=1}^d  a_{ij} p_j & \in &  \mathbb Z.
\end{eqnarray}
This identifies the group 
$\text{Aut} (W)$ 
with $d$-tuple of rational numbers 
$p = (p_j)$ defined up to $\mathbb Z^d$
such that $A p \in \mathbb Z^d$. 
Let $\rho_i$ be the $i$-th column of $A^{-1}$, 
then $A \rho_i = e_i$ 
where $e_i$ is the $i$-th standard vector. 
The group $\text{Aut} (W)$ 
is generated by the $\rho_i$-s, 
and we have 
$q = \sum_i \rho_i$
where $q = (q_j)$ is the vector 
encoding the rational degrees of $x_j$. 
The corresponding scaling operator 
\begin{eqnarray}
J_W: x_j & \mapsto & 
\text{exp}(2\pi i q_j) x_j
\end{eqnarray}
is called the exponential grading operator. 
Other than the subgroup of $\text{Aut} (W)$ 
generated by $J_W$, 
we are also interested in the subgroup 
$SL_W = SL_d \cap \text{Aut} (W)$ 
defined as follows:
\begin{eqnarray}
SL_W & = & 
\{ \gamma \in \text{Aut} (W): \prod_j \gamma_j = 1 \}. 
\end{eqnarray}
This corresponds to the condition that 
$\sum_j p_j \in \mathbb Z$ 
for $p = (p_j) \in \text{Aut} (W)$. 

We impose the condition that
the subgroup $\langle J_W \rangle$
generated by the exponential grading operator
lies in $SL_W$. 
This translates to the (generalized) Calabi-Yau condition [B2]:
\begin{eqnarray}
\sum_{j=1}^d q_j & = & 
k \in \mathbb Z_{>0}. 
\end{eqnarray}
Let $G$ be a subgroup of $\text{Aut} (W)$ 
that contains $J_W$ and is contained in $SL_W$, 
that is $\langle J_W \rangle \subset G \subset SL_W$. 
The group $G$ acts on the hypersurface 
$X_W = \{W=0 \}$
in a weighted projective space
by scaling each coordinate. 
To describe the mirror 
of the quotient $X_W/G$, 
we need the dual potential and dual group. 

The dual potential $W^\vee$ is obtained by transposing 
the exponent matrix $A$, 
i.e. 
\begin{eqnarray}
W^\vee & = & 
\sum_{i=1}^d \prod_{j=1}^d x_j^{a_{ij}}
\end{eqnarray}
It follows from the classification of [KS]
that if $W$ is non-degenerate, 
then $W^\vee$ is also non-degenerate. 
In fact, 
one observes that 
the dual potential of each type 
in (2.3)-(2.5)
is a potential of the same type. 
We also consider 
the group $\text{Aut} (W^\vee)$ 
of diagonal automorphisms 
that preserves $W^\vee$. 
If identified with (row) vectors $\bar p$ in 
$\mathbb Q^d / \mathbb Z^d$
such that $\bar p A \in \mathbb Z^d$, 
$\text{Aut} (W^\vee)$ is generated 
by the rows of $A^{-1}$. 
Similarly, we define 
the dual exponential grading operator
$J_{W^\vee}$
and the subgroup $SL_{W^\vee}$. 
Given $G$ such that 
$\langle J_W \rangle \subset G \subset SL_W$, 
there is a natural way of defining a dual group $G^\vee$
such that 
$\langle J_{W^\vee} \rangle \subset G^\vee \subset SL_{W^\vee}$
([K]). 
We will describe this duality in the language of dual lattices [B2]. 

Let $M_0$ and $N_0$ be free abelian groups 
with bases $\{ u_i \}, i =1, \cdots, d$
and $\{v_j\}, j =1, \cdots, d$. 
Define a non-degenerate integral pairing 
on these lattices by putting 
\begin{eqnarray}
u_i \cdot v_j & =  & a_{ij}, 
\end{eqnarray}
where $a_{ij}$ are the exponents in the polynomial potential $W$. 
Because the pairing is integral, 
we have 
$$
M_0 \subset N_0^\vee, \qquad
N_0 \subset M_0^\vee
$$
where $N_0^\vee$ and $M_0^\vee$ are 
the dual lattices of $N_0$ and $M_0$. 
It was shown in [B2] that 
$\text{Aut} (W)$ is naturally isomorphic to 
$M_0^\vee/ N_0$. 
Indeed, 
given $(p_j) \in \mathbb Q^d$, 
form $v = \sum_j p_j v_j$, 
then (2.9) is equivalent to 
$u_i \cdot v \in \mathbb Z$
for all $i$
which implies $v \in M_0^\vee$. 
Moreover, 
integer-valued $(p_j)$
corresponds to $v \in N_0$. 
The image of $J_W$
under this isomorphism 
can be represented by 
\begin{eqnarray}
\text{deg}^\vee & = & 
\sum_{j=1}^d q_j v_j \in M_0^\vee, 
\end{eqnarray}
then
\begin{eqnarray}
u_i \cdot \text{deg}^\vee & = & 1
\end{eqnarray}
for all $i$. 
Similarly, 
the group $\text{Aut} (W^\vee)$ is 
naturally isomorphic to 
$N_0^\vee/M_0$, 
and $J_{W^\vee}$ is represented by
$\text{deg} \in N_0^\vee$
such that 
\begin{eqnarray}
\text{deg} \cdot v_j & = & 1
\end{eqnarray}
for all $j$. 

Each subgroup $G \subset \text{Aut} (W)$
determines a suplattice $N \supset N_0$
such that $G \cong N/N_0$. 
The dual group $G^\vee \subset \text{Aut} (W^\vee)$
is defined to be 
$M/M_0$
where $M$ is the dual lattice of $N$. 
In particular, 
the dual of $\langle J_W \rangle$ is $SL_{W^\vee}$; 
the dual of $SL_W$ is $\langle J_{W^\vee} \rangle$. 
Indeed, 
for the entries of $(p_j) \in \text{Aut} (W)$ 
to add up to an integer
(the condition of $(p_j)$ being in $SL_W$) 
is equivalent to 
$(\text{deg} \cdot \sum_j p_j v_j)$
being an integer. 
It is now clear that 
if $G$ sits between $\langle J_W \rangle$ 
and $SL_W$, 
then $G^\vee$ sits between 
$\langle J_{W^\vee} \rangle$ and 
$SL_{W^\vee}$; 
the corresponding lattices then satisfy 
$\text{deg} \in M$ and 
$\text{deg}^\vee \in N$. 

The dual potential $W^\vee = 0$ defines 
a hypersurface $X_{W^\vee}$ in another 
weighted projective space. 
Berglund-H\"ubsch mirror symmetry asserts 
that $X_W/G$ and $X_{W^\vee}/G^\vee$
are mirror of each other 
for $\langle J_W \rangle \subset G \subset SL_W$.

\section{The vertex algebra $V_{{\bf 1}, {\bf 1}}$}

In this section, 
we define a vertex algebra 
associated to the above combinatorial data, 
and state some results about it. 
We mostly follow the exposition of [B1, B2]. 

Fix dual lattices $M$ and $N$
such that 
$M_0 \subset M \subset N_0^\vee$, 
$N_0 \subset N \subset M_0^\vee$, 
$\text{deg} \in M$, 
and $\text{deg}^\vee \in N$. 
First, 
we define a vertex superalgebra 
$\text{Fock}_{M \oplus N}$
which is the tensor product of the lattice vertex algebra associated to $M \oplus N$
and a vertex superalgebra generated by $2d$ fermions. 
Let 
\begin{eqnarray}
L & = & M \oplus N.
\nonumber 
\end{eqnarray}
The non-degenerate pairing between $M$ and $N$ 
extends to a non-degenerate bilinear form on $L$
where the only non-zero pairing is between an element 
of $M$ and an element of $N$. 
$L$ is thus an even lattice, 
though not positive-definite. 
Consider the 2-cocycle
\begin{eqnarray} \label{coc}
& c: L \times L \to \{ \pm 1 \} &
\nonumber
\end{eqnarray}
defined by 
\begin{eqnarray}
c((m, n), (m_1, n_1))  & = & (-1)^{m \cdot n_1}. 
\end{eqnarray}
Let $V_L$ be the lattice vertex algebra associated to $L$ and this cocycle (see e.g. [HL]). 
We use $A, B$ to distinguish modes coming from elements of $N$ and $M$. 
That is, we denote
\begin{eqnarray}
m \cdot B (z) = \sum_{k \in \mathbb Z} m \cdot B[k] z^{-k-1}, & & 
n \cdot A(z) = \sum_{k \in \mathbb Z} n \cdot A[k] z^{-k-1}
\end{eqnarray}
which have OPE: 
\begin{eqnarray} \label{ope}
m \cdot B(z) \, m_1 \cdot B(w) \sim n \cdot A(z) \, n_1 \cdot A(w) \sim 0, & &
m \cdot B(z) \, n \cdot A(w) \sim \frac{m \cdot n}{(z-w)^2}. 
\end{eqnarray}
As a vector space, 
$V_L$ is isomorphic to the direct sum of infinitely many polynomial algebras
indexed by elements of $M$ and $N$ 
in infinitely many variables 
\begin{eqnarray} 
& \oplus_{m \in M, n \in N} \mathbb C[ B[-1], B[-2], \cdots, A[-1], A[-2], \cdots] \, |m, n\rangle &
\nonumber
\end{eqnarray}
Here, 
each $B[-k], k \geq 1$ does not stand for one mode, but $d$ linearly independent ones 
corresponding to a basis of $M$; 
the same is true for $A[-k]$. 
Each $|m, n \rangle$ is annihilated by $m_1 \cdot B[k]$, $n_1 \cdot A[k]$ for $k>0$ and 
\begin{eqnarray}
m_1 \cdot B[0]  \, |m, n\rangle =
m_1 \cdot n \, |m, n \rangle, & &
n_1 \cdot A[0] \, |m, n \rangle = 
n_1 \cdot m  |m, n \rangle. 
\end{eqnarray}
We denote the vertex operators of $|m, n \rangle$ by 
$e^{\int m \cdot B  (z) + n \cdot A (z)}$. 
By definition, 
it acts on an arbitrary element of $V_L$ as follows
\begin{eqnarray} 
e^{\int m \cdot B  (z) + n \cdot A (z)} 
\prod A[\cdots] \prod B[\cdots]
|m_1, n_1 \rangle = 
(-1)^{m \cdot n_1} \text{exp}( \sum_{k <0} (m \cdot B[k] + n \cdot A[k]) \frac{z^{-k}}{-k})
& & \nonumber \\
\text{exp}( \sum_{k >0} (m \cdot B[k] + n \cdot A[k]) \frac{z^{-k}}{-k}) 
z^{m\cdot n_1+ n \cdot m_1} 
\prod A[\cdots] \prod B[\cdots] 
|m+m_1, n+n_1\rangle. & &
\nonumber
\end{eqnarray}
Let $\Lambda_L$ be the vertex super-algebra generated by the fermionic fields:
\begin{eqnarray} 
m \cdot \Phi(z) = \sum_{k \in \mathbb Z} m \cdot \Phi [k] z^{-k -1}, 
\quad & & \quad
n \cdot \Psi(z) = \sum_{k \in \mathbb Z} n \cdot \Psi [k] z^{-k} 
\nonumber
\end{eqnarray}
with OPE:
\begin{eqnarray} 
m \cdot \Phi(z) \,\, n \cdot \Psi(w) & \sim & 
\frac{ m \cdot n} {z-w}. 
\nonumber
\end{eqnarray}
As a vector space, 
$\Lambda_L$ is isomorphic to the exterior algebra
\begin{eqnarray} 
\wedge^\cdot( \oplus_{k < 0} \Phi[k] \oplus \oplus_{k \leq 0} \Psi[k]). 
\nonumber
\end{eqnarray}
Define a vertex super-algebra
\begin{eqnarray} 
\text{Fock}_{M \oplus N} & = & V_L \otimes \Lambda_L.
\nonumber
\end{eqnarray}
Consider the following two bosonic fields in 
$\text{Fock}_{M \oplus N} $
with normal ordering implicit:
\begin{eqnarray} 
J(z) & = & - \sum_{i =1}^d m_i \cdot \Phi (z) \, n_i \cdot \Psi (z) - \text{deg} \cdot B(z) +
\text{deg}^\vee \cdot A(z),
\label{jfield}
\\
L(z) & = & \sum_{i =1}^d m_i \cdot B(z) \, n_i \cdot A(z)  - \sum_{i=1}^d m_i \cdot \Phi(z) \partial_z \, n_i \cdot \Psi(z) - \partial_z \text{deg} \cdot B(z),
\label{lfield}
\end{eqnarray}
where $\{ m_i \}$, $\{n_i\}$ are dual bases in $M$ and $N$. 
Write 
\begin{eqnarray} 
J(z) = \sum_{k \in \mathbb Z} J[k] z^{-k-1}, \quad & & \quad
L(z) = \sum_{k \in \mathbb Z} L[k] z^{-k-2}.
\nonumber
\end{eqnarray}
The eigenvalues of $J[0]$ and $L[0]$ 
equip $\text{Fock}_{M \oplus N}$  with a double grading. 
Explicitly, 
given an element 
\begin{eqnarray} 
\prod A[\cdots] \prod B[\cdots] \prod \Phi[\cdots] \prod \Psi[\cdots] |m, n\rangle
& \in & \text{Fock}_{M \oplus N}, 
\nonumber
\end{eqnarray}
$J[0]$ counts the number of occurrences of $\Psi$ 
minus the number of occurrences of $\Phi$, 
plus $(\text{deg}^\vee \cdot m -\text{deg} \cdot n )$, 
while $L[0]$ counts the opposite of the sum of indices in $[ \,\, ]$, 
plus $m \cdot n + \text{deg} \cdot n$.

Denote by $\triangle$ the set $(u_i)$ and by $\triangle^\vee$ the set $(v_j)$. 
Define the cones $K_M$ in $M$ and $K_N$ in $N$ by 
\begin{eqnarray} 
K_M : = M \cap \sum_i \mathbb Q_{\geq 0} u_i \quad & & \quad
K_N : = N \cap \sum_j \mathbb Q_{\geq 0} v_j.
\nonumber
\end{eqnarray}
Since $u_i \cdot \text{deg}^\vee = 1$ for all $i$
and $\text{deg}^\vee \in N$, 
$(u_i)$ are the primitive generators of the rays in $K_M$, 
so are $(v_j)$ primitive generators of the rays in $K_N$. 
Consider the following operators
\begin{eqnarray} 
D_{1, 0} & = & 
\text{Res}_{z=0} \sum_{m \in \triangle} m \cdot \Phi (z) e^{\int m \cdot B(z)}, 
\nonumber \\
D_{0, 1} & = &
\text{Res}_{z=0} \sum_{n \in \triangle^\vee} n \cdot \Psi(z) e^{\int n \cdot A(z)}, 
\nonumber \\
D_{1, 1} & = & D_{1, 0} + D_{0, 1}.
\nonumber 
\end{eqnarray}
They all commute with $J[0]$ and $L[0]$. 
It is direct to check that 
$D_{1, 0}$ and $D_{0, 1}$ are both differentials, 
and they anticommute,
hence $D_{1, 1}$ is also a differential.
Introduce a bi-grading on $\text{Fock}_{M \oplus N}$
by the eigenvalues of 
\begin{eqnarray} 
\text{deg}^\vee \cdot A[0] & \quad \text{and} \quad & 
\text{deg} \cdot B[0]. 
\nonumber
\end{eqnarray}
Then $D_{1, 0}$ and $D_{0, 1}$ 
change the $(\text{deg}^\vee \cdot A[0], \text{deg} \cdot B[0])$-grading
by $(1, 0)$ and $(0, 1)$ respectively,
hence $(\text{Fock}_{M \oplus N}, D_{1, 0}, D_{0, 1})$ 
form a double complex. 

Define the vertex super-algebra $V_{1, 1}$ as the cohomology of 
$\text{Fock}_{M \oplus N}$ with respect to the total differential $D_{1, 1}$. 
The $(J[0], L[0])$-grading on Fock descends to the cohomology of 
the operators $D_{1, 0}$, $D_{0, 1}$, and $D_{1, 1}$.

We need the following results from [B1], [B2]. 

\begin{prop}{[B2, Theorem 5.2.3]}
The cohomology of $\text{Fock}_{M \oplus N}$ 
with respect to $D_{1, 1}$ is equal to the cohomology of 
$\text{Fock}_{M \oplus K_N}$ with respect to $D_{1, 1}$. 
\end{prop}

\begin{prop}{[B2, Theorem 6.2.1]}
For fixed eigenvalues of $J[0]$ and $L[0]$, 
the corresponding eigenspace in $V_{1, 1}$ is finite-dimensional.
\end{prop}

Theorem 6.2.1 of [B2] is in fact much stronger. 
We are interested in computing the supertrace of the operator $y^{J[0]} q^{L[0]}$ 
on $V_{1, 1}$. 
Supertrace means that we subtract the dimension of the odd part from the dimension of the even part of the corresponding eigenspaces. 
By the previous proposition, 
this gives a well-defined double series in $y$ and $q$. 
To compute this invariant, 
we first describe the cohomology of 
$\text{Fock}_{M \oplus K_N}$ with respect to $D_{0, 1}$.


\begin{prop}{[B1, Proposition 9.3]} \label{MKNDg}
Denote by $(v_j^\vee)$ the dual basis of $(v_j)$ in $N_0^\vee$.
For each $i$, 
define 
\begin{eqnarray}
b^i (z) = e^{\int v_i^\vee \cdot B(z)}, \quad 
\phi^i (z) = (v_i^\vee \cdot \Phi(z)) e^{\int v_i^\vee \cdot B(z)}, \quad
\psi_i(z) = (v_i \cdot \Psi(z)) e^{ - \int v_i^\vee \cdot B(z)} 
\nonumber \\
a_i (z) = :(v_i \cdot A(z)) e^{- \int v_i^\vee \cdot B(z)}: 
+ : (v_i^\vee \cdot \Phi(z)) (v_i \cdot \Psi(z)) e^{- \int v_i^\vee \cdot B(z)} :
\nonumber
\end{eqnarray}
These fields generate a vertex subalgebra $\mathcal{VA}_{K_N, N_0^\vee}$
inside $\text{Fock}_{N_0^\vee \oplus 0}$. 
Consider all elements from $\mathcal{VA}_{K_N, N_0^\vee}$
whose $A[0]$ eigenvalues lie in $M$. 
Denote the resulting algebra by $\mathcal{VA}_{K_N, M}$. 

Let $Box (K_N)$ be the set of all elements $n \in K_N$
such that $n - v_j \notin K_N$ for all $j$. 
Equivalently, 
$Box (K_N) = \{ \sum_{j} p_j v_j \in K_N: 0 \leq p_j <1 \}$. 
For every $n \in Box (K_N)$, 
consider the following set of elements of Fock$_{M \oplus n}$. 
For every $v = \prod A[\cdots] \prod B[\cdots] \prod \Phi[\cdots] \prod \Psi[\cdots] 
|m, 0 \rangle$ that lies in 
$\mathcal{VA}_{K_N, M} \subset \text{Fock}_{M \oplus 0}$, 
consider $v' = \prod A[\cdots] \prod B[\cdots] \prod \Phi[\cdots] \prod \Psi[\cdots] 
|m, n \rangle$ which is obtained by applying the same modes of 
$A, B, \Phi$, and $\Psi$ to $|m, n \rangle$ 
instead of $|m, 0\rangle$. 
We denote this space by 
$\mathcal{VA}_{K_N, M}^{(n)}$. 

Then the cohomology of $\text{Fock}_{M \oplus K_N}$ with respect to 
$D_{0, 1}$ is equal to 
\begin{eqnarray} \label{dec}
\text{Fock}_{M \oplus K_N} / D_{0, 1} & = & 
\oplus_{n \in \text{Box}(K_N)}
\mathcal{VA}_{K_N, M}^{(n)}. 
\end{eqnarray}
\end{prop}

Basically, 
the differential $D_{0, 1}$ preserves $\text{Fock}_{M \oplus (n + \sum_j \mathbb N v_j)}$
for each $n \in Box (K_N)$.
The cohomology $\text{Fock}_{M \oplus (n +  \sum_j \mathbb N v_j )}$
with respect to $D_{0, 1}$
for various $n$ all look like the cohomology of 
$\text{Fock}_{M \oplus \sum_j \mathbb N v_j}$ with respect to $D_{0, 1}$. 

\section{Elliptic genera of BH-models}

We aim to derive a formula for the supertrace of $y^{J[0]} q^{L[0]}$
on the cohomology of $\text{Fock}_{M \oplus K_N}$ with respect to $D_{0, 1}$. 
By (\ref{dec}), 
it is sufficient to compute the supertrace of $y^{J[0]} q^{L[0]}$
on each summand 
$\mathcal{VA}_{K_N, M}^{(n)}$. 
For $n = 0$, 
$\mathcal{VA}_{K_N, M}$ consists of those from 
$\mathcal{VA}_{K_N, N_0^\vee}$
whose $A[0]$-eigenvalue lies in $M$. 
The vertex superalgebra 
$\mathcal{VA}_{K_N, N_0^\vee}$
is generated by 
the $2d$ bosonic fields $b^i(z)$, $a_i(z)$, 
and $2d$ fermionic fields
$\phi^i(z)$, and $\psi_i(z)$ 
with the following OPE
\begin{eqnarray}
a_i (z) b^j (w) \sim \frac{ \delta_{ij} }{z-w}, 
& & 
\phi^i (z) \psi_j (w) \sim \frac{\delta_{ij}}{z-w}
\nonumber
\end{eqnarray}
(all other OPEs vanish). 
These fields (by the state-field correspondence)
admit the following
$(J[0], L[0])$-eigenvalues:
\begin{eqnarray}
                            & J[0]                                       & L[0]   \nonumber\\
a_i(z) &   - \text{deg}^\vee \cdot v_i^\vee       &  1 \nonumber \\
b^i(z) &    \text{deg}^\vee \cdot v_i^\vee     &  0  \nonumber \\
\phi^i(z) &  \text{deg}^\vee \cdot v_i^\vee - 1 & 1 \nonumber \\
\psi_i(z) & - \text{deg}^\vee \cdot v_i^\vee  +1  &  0 \nonumber 
\end{eqnarray}
Note that 
$\text{deg}^\vee = \sum_{j=1}^d q_j v_j$ (see (2.15)), 
hence $\text{deg}^\vee \cdot v_i^\vee = q_i$.
The previous table becomes
\begin{eqnarray}
                            & J[0]                                       & L[0]   \nonumber\\
a_i(z) &  -q_i    &  1 \nonumber \\
b^i(z) &   q_i  &  0  \nonumber \\
\phi^i(z) &  q_i - 1 &  1  \nonumber \\
\psi_i(z) &  -q_i +1  & 0 \nonumber 
\end{eqnarray}
Now, 
the supertrace of 
$y^{J[0]} q^{L[0]}$ on $\mathcal{VA}_{K_N, N_0^\vee}$ can be computed as follows:
\begin{eqnarray}
\text{ST}_{\mathcal{VA}_{K_N, N_0^\vee}} y^{J[0]} q^{L[0]} = 
\prod_{i=1}^d 
\frac{
\prod_{k \geq 0} 
(1- y^{ - q_i+1} q^{k }) 
\prod_{k \geq 1}
(1 - y^{q_i-1} q^{ k})
}{
\prod_{k \geq 0}
(1- y^{q_i} q^{ k})
\prod_{k \geq 1}
(1- y^{-q_i} q^{k})
}
\end{eqnarray}

The infinite products on the numerator
come from the modes of $\phi^i(z)$ and $\psi_i(z)$; 
the products on the denominator 
come from the modes of $a_i(z)$ and $b^i(z)$. 
This expression involves rational powers of $y$ and $q$. 
To extract the supertrace of $y^{J[0]} q^{L[0]}$ 
on the subalgebra $\mathcal{VA}_{K_N, M} \subset \mathcal{VA}_{K_N, N_0^\vee}$ 
from (4.1), 
we need to insert certain roots of $1$ to eliminate the terms 
contributed by those in $ \mathcal{VA}_{K_N, N_0^\vee}$ 
whose $A[0]$-eigenvalue lies outside $M$. 

Recall the finite abelian group $G = N/N_0$. 
As a set, $G$ is isomorphic to Box$(K_N)$. 
For any $n_1 \in N$, 
we define 
\begin{eqnarray}
\theta_j( n_1 ) & = &  v_j^\vee \cdot n_1.
\end{eqnarray}
Consider the group algebra 
$\mathbb C[N_0^\vee] = \mathbb C[x_1^{\pm 1}, \cdots, x_d^{\pm 1}]$. 
The variables $(x_j)$ correspond to the basis $(v_j^\vee)$ of $N_0^\vee$. 
The group 
$G$ acts on $\mathbb C[N_0^\vee]$ as follows:
for any $n_1 \in N/N_0$, $m \in N_0^\vee$, 
we have 
\begin{eqnarray}
{n_1}_\cdot [m] & = & \text{exp} (2 \pi i ( n_1 \cdot m) ) \, [m].
\nonumber
\end{eqnarray}
Then the $G$-invariant of $\mathbb C[N_0^\vee]$ is $\mathbb C[M]$, 
i.e. $\mathbb C[N_0^\vee]^G = \mathbb C[M]$. 
There is an "averaging over $G$" operation 
from 
$\mathbb C[N_0^\vee]$ to $\mathbb C[M]$
that we can use to obtain the supertrace of 
$y^{J[0]} q^{L[0]}$ on $\mathcal{VA}_{K_N, M}$, 
that is to insert 
$\frac{1}{|G|} \sum_{n_1 \in G} \text{exp} ( 2 \pi i \, m \cdot n_1 )$ 
in front of the term contributed by 
element
$\prod A[\cdots] \prod B[\cdots] \prod \Phi[\cdots] \prod \Psi[\cdots] |m, 0 \rangle \in 
\mathcal{VA}_{K_N, N_0^\vee}$. 
Hence, 
we have
\begin{eqnarray}
& & 
\text{ST}_{\mathcal{VA}_{K_N, M}} y^{J[0]} q^{L[0]} 
 \\
& = & 
\frac{1}{|G|} \sum_{n_1 \in G} 
\prod_{j=1}^d 
\frac{
\prod_{k \geq 0} 
(1- y^{ - q_j+1} q^{k } e^{ -2 \pi i \theta_j (n_1)} )
\prod_{k \geq 1}
(1 - y^{q_j-1} q^{ k} e^{2 \pi i \theta_j (n_1)})
}{
\prod_{k \geq 0}
(1- y^{q_j} q^{ k} e^{ 2 \pi i \theta_j (n_1)})
\prod_{k \geq 1}
(1- y^{-q_j} q^{ k} e^{ -2 \pi i \theta_j (n_1)})
}
\nonumber
\end{eqnarray}
In general for $n \in \text{Box}(K_N)$, 
the supertrace of $y^{J[0]}q^{L[0]}$ on 
$\mathcal{VA}_{K_N, M}^{(n)}$
is given by 
\begin{eqnarray} \label{t1}
& & 
\text{ST}_{\mathcal{VA}_{K_N, M}^{(n)}}  y^{J[0]} q^{L[0]} 
 \\
& = & 
(y^{-1} q)^{\text{deg} \cdot n} 
\frac{1}{|G|} \sum_{n_1 \in G} 
\prod_{j=1}^d 
\frac{
\prod_{k \geq 0} 
(1- y^{ - q_j+1} q^{k - \theta_j(n) } e^{ -2 \pi i \theta_j (n_1)} )
\prod_{k \geq 1}
(1 - y^{q_j-1} q^{ k +  \theta_j(n) } e^{2 \pi i \theta_j (n_1)})
}{
\prod_{k \geq 0}
(1- y^{q_j} q^{ k +  \theta_j(n) } e^{ 2 \pi i \theta_j (n_1)})
\prod_{k \geq 1}
(1- y^{-q_j} q^{ k -  \theta_j(n)} e^{ -2 \pi i \theta_j (n_1)})
}
\nonumber
\end{eqnarray}
The above is understood as a Laurent series in $y$, $q$ with rational powers and non-negative powers of $q$. 
Indeed, 
each $\theta_j(n)$ lies in $[0, 1)$. 
The powers of $q$ that appear on the denominator are all non-negative, 
hence when the reciprocal of the denominator terms are expressed as a power series, 
only non-negative powers of $q$ appear. 
On the numerator, 
the only term that could have a negative power of $q$
is when $k=0$ in the first infinite product.
However, 
we have an extra term $q^{\text{deg} \cdot n}$ in the front, 
and the fact that $\text{deg} \cdot n = \sum_j \theta_j(n)$ takes care of it. 
This double series converges absolutely 
when $|q| < | y^{q_j} q^{\theta_j(n)}| < 1$ for all $j$. 
In fact, 
we can write it in terms of the theta function. 
Let 
\begin{eqnarray} 
\Theta(\nu, \tau) & = &
i q^{\frac{1}{8}} e^{- i \pi \nu} (1 - e^{i 2 \pi \nu}) 
\prod_{n=1}^\infty 
(1- q^n) (1 - q^n e^{i 2 \pi \nu} ) (1 - q^n e^{- i 2 \pi \nu})
\end{eqnarray}
where 
$q = e^{i 2 \pi \tau}$
be Jacobi's theta function. 
It is a holomorphic function for $\nu \in \mathbb C$,  $\tau \in H$
where $H$ is the upper-half plane. 
If we fix $\tau \in H$, 
then $\Theta(\nu, \tau)$, 
as a function of $\nu$, 
has single zeroes at all the lattice points in $\mathbb Z \tau + \mathbb Z$. 
Multiplying both the numerator and the denominator of (\ref{t1}) by 
$\prod_{n =1}^\infty (1- q^n)$, 
we obtain 
\begin{eqnarray}
& & 
\text{ST}_{\mathcal{VA}_{K_N, M}^{(n)}}  y^{J[0]} q^{L[0]} 
 \\
& = & 
q^{\text{deg} \cdot n}
\frac{1}{|G|} 
\sum_{n_1 \in G}
\prod_{j=1}^d
y^{- \theta_j(n)}
\frac{
e^{i\pi \{ (1- q_j) z - \theta_j(n) \tau - \theta_j(n_1) \}}
\Theta( (1-q_j) z - \theta_j (n) \tau - \theta_j(n_1), \tau )
}{
e^{i \pi \{ q_j z + \theta_j(n) \tau + \theta_j(n_1) \}} 
\Theta( q_j z + \theta_j(n) \tau + \theta_j(n_1), \tau)
}
\nonumber \\
& = &
q^{\text{deg} \cdot n}
\frac{1}{|G|} 
\sum_{n_1 \in G}
e^{i \pi z (d - 2 \sum_j q_j)} e^{- i 2 \pi \tau \sum_j  \theta_j(n)} e^{-i 2 \pi \sum_j \theta_j(n_1)}
\prod_{j=1}^d
e^{- i 2 \pi z \theta_j(n)}
\frac{
\Theta(\cdots)
}{
\Theta(\cdots)
}
\nonumber
\end{eqnarray}
where $y = e^{i 2 \pi z}$, $q = e^{i 2 \pi \tau}$. 
Note that 
$q^{\text{deg} \cdot n}$ cancels with
$e^{- i 2 \pi \tau \sum_j  \theta_j(n)}$.
Moreover, 
$\sum_j \theta_j(n_1) = \text{deg} \cdot n_1 \in \mathbb Z$
because 
$\text{deg} \in M$, $n_1 \in N$, and $M$ and $N$ are dual lattices, 
hence 
$e^{-i 2 \pi \sum_j \theta_j(n_1)} =1$. 
We have 
\begin{eqnarray} \label{t2}
& & 
\text{ST}_{\mathcal{VA}_{K_N, M}^{(n)}}  y^{J[0]} q^{L[0]} 
 \\
& = & 
y^{\frac{1}{2} (d - 2 \sum_j q_j)}
\frac{1}{|G|} 
\sum_{n_1 \in G}
\prod_{j=1}^d
e^{- i 2 \pi z \theta_j(n)}
\frac{
\Theta( (1-q_j) z - \theta_j (n) \tau - \theta_j(n_1), \tau )
}{
\Theta( q_j z + \theta_j(n) \tau + \theta_j(n_1), \tau)
}. 
\nonumber
\end{eqnarray}

Note that 
\begin{eqnarray}
q_j & = & v_j^\vee \cdot \text{deg}^\vee 
= \theta_j (\text{deg}^\vee)
\end{eqnarray}
for $\text{deg}^\vee \in N$. 
The number 
\begin{eqnarray}
\hat c & = & 
 d - 2 \sum_j q_j = d - 2 \,\, \text{deg} \cdot \text{deg}^\vee
\end{eqnarray}
is the central charge of the $N=2$ structure on $V_{1, 1}$
([B2]). 

\begin{theorem}
The elliptic genus of the Berglund-H\"ubsch model $W/G$
defined in [BHe] is equal to 
$y^{- \frac{1}{2} \hat c} \text{SuperTrace}_{V_{1, 1}} y^{J[0]} q^{L[0]}$.
\end{theorem}

\begin{proof}
Consider the double complex 
$(\text{Fock}_{M \oplus K_N}, D_{1, 0}, D_{0, 1})$
with bi-grading $(\text{deg}^\vee \cdot A[0], \text{deg} \cdot B[0])$. 
It lies in the upper half plane because 
$\text{deg} \cdot B[0]$ has non-negative eigenvalues on $K_N$.  
The vertex algebra $V_{1, 1}$ is the cohomology of 
the total complex. 
Consider the filtration of the total complex such that the $E^0$ terms 
of the associated spectral sequence is the cohomology of $\text{Fock}_{M \oplus K_N}$
with respect to the (vertical) differential $D_{0, 1}$. 
The filtration is bounded below and exhaustive, 
so the spectral sequence converges to the cohomology of the total complex. 
The cohomology of 
$\text{Fock}_{M \oplus K_N}$ with respect to $D_{0, 1}$
is described in Proposition \ref{MKNDg}. 
The $\text{deg} \cdot B[0]$-grading on $\text{Fock}_{M \oplus K_N}/D_{0, 1}$ 
is bounded by $d$. 
Indeed, 
each summand $\mathcal{VA}_{K_N, M}^{(n)}$ in (\ref{dec})
has the $\text{deg} \cdot B[0]$-grading equal to $\text{deg} \cdot n$
which is $<d$ because $n$ lies in the Box$(K_N)$. 
Hence, 
the spectral sequence degenerates after finitely many steps. 
Since the differentials of the spectral sequence change parity and 
commute with $J[0]$ and $L[0]$, 
they have no effect on the supertrace. 
We have 
\begin{eqnarray}
\text{SuperTrace}_{V_{1, 1}} y^{J[0]} q^{L[0]}
& =  &
\text{SuperTrace}_{\text{Fock}_{M \oplus K_N}/D_{0, 1}} y^{J[0]} q^{L[0]}. 
\nonumber
\end{eqnarray}
Finally, 
we sum up (\ref{t2}) over $n \in \text{Box}(K_N) \cong G$. 
When multiplied with $y^{- \frac{1}{2} \hat c}$, 
it matches with the formulae (2.6), (2.7), and (2.14) of [BHe]. 
\end{proof}

We denote the elliptic genus of $W/G$ by 
\begin{eqnarray}
Ell(W/G, z, \tau) & = & 
\frac{1}{|G|} 
\sum_{n, n_1 \in G}
\prod_{j=1}^d
e^{- i 2 \pi z \theta_j(n)}
\frac{
\Theta( (1-q_j) z - \theta_j (n) \tau - \theta_j(n_1), \tau )
}{
\Theta( q_j z + \theta_j(n) \tau + \theta_j(n_1), \tau)
}. 
\label{eg}
\end{eqnarray}
Our next goal is to show that this is a weak Jacobi form. 
First, we establish the holomorphicity. 

\begin{theorem}
$Ell(W/G, z, \tau)$ is a holomorphic function of two variables for all 
$z \in \mathbb C$, $\tau \in H$. 
\end{theorem}

\begin{proof}
We will show explicitly with appeal to the classification of non-degenerate potentials
that the zeroes of the theta functions on the denominator of (\ref{t2})
cancel with (some of) the zeroes of the theta functions on the numerator. 
Then it follows that the double series (\ref{t1}) converge absolutely 
to holomorphic functions for all 
$y \in \mathbb C^*, |q| <1$. 

Any non-degenerate potential $W$ is a sum of decoupled potentials of three types: 
Fermat, loop, and chain ([KS]). 
It is sufficient to prove the holomorphicity for $W$ of each type. 
If $W = x^a$, $a >2$ is of the Fermat type, 
then $q = 1/a$
and $\theta(n) \in (1/a) \mathbb Z$
for all $n \in G = \mathbb Z_a$. 
The zeroes of 
$\Theta(q z + \theta(n) \tau + \theta(n_1), \tau)$ 
correspond to those $z, \tau$ such that 
\begin{eqnarray}
q z + \theta(n) \tau + \theta(n_1) \in \mathbb Z \tau + \mathbb Z. 
\nonumber
\end{eqnarray}
When this is true, 
multiplying with $a-1$, 
we get 
\begin{eqnarray}
(1- q) z + (a-1) \theta(n) \tau + (a-1) \theta(n_1) \in \mathbb Z (a-1) \tau + \mathbb Z (a-1).
\nonumber
\end{eqnarray}
Since $a \theta(n), a \theta(n_1) \in \mathbb Z$, 
it follows that 
$(1-q) z - \theta(n) \tau - \theta(n_1) \in \mathbb Z \tau + \mathbb Z$, 
hence they are also zeroes of the numerator. 

If $W = x_1^{a_1} x_2 + x_2^{a_2} x_3 + \cdots + x_{k-1}^{a_{k-1}} x_k + x_k^{a_k} x_1$
is of the loop type, 
then we have 
\begin{eqnarray}
a_i q_i + q_{i+1} = 1 & \text{  for  } & 1 \leq i \leq k-1, 
\nonumber \\ 
a_k q_k + q_1 =1    &                        &
\nonumber \\
a_i \theta_i(n) + \theta_{i+1} (n) \in \mathbb Z                    & \text{for} & 
1 \leq i \leq k-1, n \in \text{Box} (K_N)
\nonumber \\
a_k \theta_k (n) + \theta_1(n) \in \mathbb Z     & \text{for} &   n \in \text{Box} (K_N). 
\nonumber 
\end{eqnarray}
The same arguments apply as in the Fermat case: 
the zeroes of 
$\Theta(q_j z + \theta_j(n) \tau + \theta_j(n_1), \tau)$
cancel with the zeroes of 
$\Theta((1-q_{j+1}) z - \theta_{j+1}(n) \tau - \theta_{j+1} (n_1), \tau)$
for $1 \leq j \leq k-1$, 
and the zeroes of 
$\Theta(q_k z + \theta_k(n) \tau + \theta_k(n_1), \tau)$
cancel with the zeroes of 
$\Theta((1-q_1) z - \theta_{1}(n) \tau - \theta_{1} (n_1), \tau)$.

If $W = x_1^{a_1} x_2 + x_2^{a_2} x_3 + \cdots + x_{k-1}^{a_{k-1}} x_k + x_k^{a_k}$
is of the chain type, 
then 
\begin{eqnarray}
a_i q_i + q_{i+1} = 1 & \text{  for  } & 1 \leq i \leq k-1, 
\nonumber \\ 
a_k q_k  =1    &                        &
\nonumber \\
a_i \theta_i(n) + \theta_{i+1} (n) \in \mathbb Z                    & \text{for} & 
1 \leq i \leq k-1, n \in \text{Box} (K_N)
\nonumber \\
a_k \theta_k (n) \in \mathbb Z     & \text{for} &   n \in \text{Box} (K_N). 
\nonumber 
\end{eqnarray}
The previous argument fails to cancel the zeroes of 
$\Theta(q_k z + \theta_k(n) \tau + \theta_k(n_1), \tau)$. 
Instead, 
we will "distribute" its zeroes to each term of the numerator. 
Here is the mechanism of how this works, isolated. 
Suppose we have a ratio of theta functions
\begin{eqnarray} \label{ratio}
\frac{1}{\Theta( \frac{k}{ml} z + \alpha_3 \tau + \beta_3, \tau)} \,\, 
\frac{\Theta( \frac{k}{m} z + \alpha_1 \tau + \beta_1, \tau)}
{\Theta( \frac{1}{m} z + \alpha_2 \tau + \beta_2, \tau)}
\end{eqnarray}
where $m, k, l$ are integers, $\alpha_i$, $\beta_i$ are rational numbers, 
$(m, k) =1$, 
and 
$m \alpha_2 \in \mathbb Z$, 
$m \beta_2  \in \mathbb Z$. 
Moreover, 
$k \alpha_2 \equiv \alpha_1$ (mod $\mathbb Z$), 
$k \beta_2 \equiv \beta_1$ (mod $\mathbb Z$), 
$l \alpha_3 \equiv \alpha_1$ (mod $\mathbb Z$), 
$l \beta_3 \equiv \beta_1$ (mod $\mathbb Z$). 
The zeroes of the first theta function on the denominator 
lie on the lines 
\begin{eqnarray}
\frac{k}{ml} z + \alpha_3 \tau + \beta_3 = p \tau + q, & & 
p, q \in \mathbb Z
\nonumber
\end{eqnarray}
which is equivalent to 
\begin{eqnarray} \label{l1}
\frac{k}{m} z + l \alpha_3 \tau + l \beta_3 = p  l \tau + q l, & & 
p, q \in \mathbb Z. 
\end{eqnarray}
By assumption, 
this family of lines belong to the set of lines containing the zeroes of the numerator. 
Similarly, 
the second theta function on the denominator 
has zeroes on the lines 
\begin{eqnarray} \label{l2}
\frac{k}{m} z + k \alpha_2 \tau + k \beta_2 = p'  k \tau + q' k, & & 
p', q' \in \mathbb Z. 
\end{eqnarray}
These lines again coincide with some of the lines containing the zeroes of the numerator. 
If the two families of lines (\ref{l1}) and (\ref{l2}) have no intersection, 
then the zeroes of the denominator are all cancelled by the zeroes from the numerator, 
the ratio is therefore holomorphic. 
Otherwise, 
we have 
$\text{gcd} (k, l) | (k \alpha_2 - l \alpha_3)$, 
$\text{gcd} (k, l) | (k \beta_2 - l \beta_3)$. 
The lines in (\ref{l2}) that are not "eliminated" by the lines from the numerator are those
with $(p', q')$ such that 
\begin{eqnarray} \label{p'q'}
l | (p' k - k \alpha_2 + l \alpha_3),  & & 
l | (q' k - k \beta_2 + l \beta_3). 
\end{eqnarray}
Fix such a pair $(p', q')$, 
then every other such pair $(p'', q'')$ satisfy 
that 
\begin{eqnarray}
p'' - p', q'' - q' \in \frac{l}{\text{gcd} (k, l)} \mathbb Z. 
\nonumber
\end{eqnarray}
Hence, we have the following lines remaining
\begin{eqnarray}
\frac{1}{m} z + \alpha_2 \tau + \beta_2 = ( p' +  \frac{l}{\text{gcd} (k, l)} s ) \tau + (q' +  \frac{l}{\text{gcd} (k, l)} t), & & s, t \in \mathbb Z
\nonumber
\end{eqnarray}
or 
\begin{eqnarray}
\frac{ \text{gcd}(k, l) }{m l} z + \frac{ \text{gcd}(k, l) }{l} (\alpha_2 - p') \tau
+ \frac{ \text{gcd}(k, l) }{l} (\beta_2 - q') 
= s \tau + t, 
& & s, t \in \mathbb Z. 
\nonumber
\end{eqnarray}
Set 
\begin{eqnarray}
m^{\text{new}} = \frac{ml }{\text{gcd}(k, l)}, \qquad 
\alpha_2^{\text{new}} =  \frac{ \text{gcd}(k, l) }{l} (\alpha_2 - p'), \qquad 
\beta_2^{\text{new}} = \frac{ \text{gcd}(k, l) }{l} (\beta_2 - q'),
\nonumber
\end{eqnarray}
then it is clear that 
$m^{\text{new}} \alpha_2^{\text{new}}$,  
$m^{\text{new}} \beta_2^{\text{new}} \in \mathbb Z$. 
The ratio (\ref{ratio}) has 
the same poles as 
\begin{eqnarray} \label{x1}
\frac{1}{\Theta(\frac{1}{m^{\text{new}}} z + \alpha_2^{\text{new}} \tau + \beta_2^{\text{new}}, \tau)}.
\end{eqnarray}
Finally, 
consider 
$\Theta( (1- \frac{k}{ml}) z - \alpha_3 \tau - \beta_3, \tau)$. 
We have 
\begin{eqnarray} \label{x2}
1 - \frac{k}{ml} = \frac{ \frac{m l}{\text{gcd}(k, l)} - \frac{k}{\text{gcd}(k, l)} } {\frac{m l}{\text{gcd}(k, l)}}
= : \frac{k^{\text{new}}}{m^{\text{new}}}, 
\qquad 
(k^{\text{new}}, m^{\text{new}}) =1.
\end{eqnarray}
Moreover, 
\begin{eqnarray} \label{x3}
k^{\text{new}} \alpha_2^{\text{new}} = m (\alpha_2 -p') - \frac{k (\alpha_2 -p') }{l}
\equiv - \alpha_3 (\text{   mod  } \mathbb Z)
\end{eqnarray}
because $m \alpha_2 \in \mathbb Z$ by assumption 
and (\ref{p'q'}). 
Similarly, 
we have 
\begin{eqnarray} \label{x4}
k^{\text{new}} \beta_2^{\text{new}} = m (\beta_2 -q') - \frac{k (\beta_2 -q') }{l}
\equiv - \beta_3 (\text{   mod  } \mathbb Z)
\end{eqnarray}
All of (\ref{x1})-(\ref{x4})
will be the beginning of another round of the same arguments. 

Now, let us see how this applies to prove the holomorphicity of 
$Ell (W/G, z, \tau)$ 
for $W$ of the chain type. 
We need to examine
\begin{eqnarray}
\cdots \cdots
\frac{\Theta((1-q_{k-1}) z - \theta_{k-1} (n) \tau - \theta_{k-1} (n_1), \tau)}
{\Theta(q_{k-1} z + \theta_{k-1} (n) \tau + \theta_{k-1} (n_1), \tau)}
\frac{\Theta((1-q_k ) z - \theta_k (n) \tau - \theta_k (n_1), \tau)}
{\Theta(q_k z + \theta_k (n) \tau + \theta_k (n_1), \tau)}.
\nonumber
\end{eqnarray}
Recall that 
$q_k = \frac{1}{a_k}$ and 
$a_k \theta_k(n), a_k \theta_k(n_1) \in \mathbb Z$. 
It is clear that
the ratio of the last theta function on the numerator 
and the last two theta functions on the denominator 
satisfy the assumptions of the above discussion. 
We apply the above arguments finitely many times from right to left, 
in the end, 
we are able to cancel all zeroes of the denominator. 
\end{proof}

\begin{theorem} \label{modularity}
$Ell(W/G, z, \tau)$ is a weak Jacobi form of weight $0$ and index $\frac{\hat c}{2}$. 
\end{theorem}

Weak here means that it obeys the transformation laws of the Jacobi forms, 
however at the cusp we require that only non-negative powers of $q$ appear ([EZ]). 
Also, 
when $\hat c$ is odd, 
the definition of the Jacobi form is modified to allow a character. 


\begin{proof}
The condition at the cusp holds because (\ref{t1}) has no negative powers of $q$. 
It is now enough to verify the following modular properties of 
$Ell(W/G, z, \tau)$: 
\begin{eqnarray}
Ell(W/G, z, \tau +1 ) & = & Ell(W/G, z, \tau) \label{e1}
\\
Ell(W/G, z + 1, \tau) & = & (-1)^{\hat c} Ell(W/G, z, \tau) \label{e2}
\\
Ell(W/G, z + \tau, \tau) & = & (-1)^{\hat c} e^{- i \pi \hat c (\tau + 2 z)} Ell(W/G, z, \tau) \label{e3}
\\
Ell(W/G, \frac{z}{\tau}, -\frac{1}{\tau}) & = & e^{\frac{i \pi {\hat c}  z^2}{\tau}} Ell(W/G, z, \tau) \label{e4}
\end{eqnarray}
We need the following identities of the theta function:
\begin{eqnarray}
\Theta(\nu, \tau+1) &  = & \Theta(\nu, \tau) \label{th1}
\\
\Theta(\nu +1, \tau)  & = &  - \Theta(\nu, \tau) \label{th2}
\\
\Theta(\nu + \tau, \tau) & = & - e^{-i 2 \pi \nu - i \pi \tau} \Theta(\nu, \tau)  \label{th3}
\\
\Theta(\frac{\nu}{\tau}, - \frac{1}{\tau}) & = & 
- i \sqrt{\frac{\tau}{i}} e^{ \frac{i \pi \nu^2}{\tau} } \Theta(\nu, \tau).  \label{th4}
\end{eqnarray}
(\ref{e1}) follows from (\ref{th1}) and the change of variable
$n n_1 \to n_1$ in (\ref{eg}). 
(\ref{e2}) follows from (\ref{th2}), 
$e^{-i 2 \pi \sum_j \theta_j(n) } = 1$
(because $\sum_j \theta_j(n) = \text{deg} \cdot n \in \mathbb Z$), 
and the change of variable $\text{deg}^\vee\, n_1 \to n_1$
(note that $q_j = \theta_j (\text{deg}^\vee)$ and
$\text{deg}^\vee \in G$). 
Also, 
$\sum_j q_j \in \mathbb Z$, 
hence $\hat c = d$ (mod $2$). 
(\ref{e3}) follows from (\ref{th3}), 
$\sum_j \theta_j(n_1) \in \mathbb Z$, 
and the change of variable
$\text{deg}^\vee \,\, n \to n$. 
(\ref{e4}) follows from (\ref{th4}) and change of variables
$n \to n_1^{-1}$, $n_1 \to n$. 
\end{proof}

\begin{remark}
It was shown in [BHe] that the elliptic genus $Ell(W/G)$ 
satisfy modular transformation properties 
with respect to $(z, \tau) \to (z, \tau+1)$, 
$(z, \tau) \to (\frac{z}{\tau}, -\frac{1}{\tau})$
(same as here), and
$(z, \tau) \to (z + L, \tau)$, 
$(z, \tau) \to (z+ L\tau, \tau)$, 
where $L$ is the smallest integer such that $g^L = \text{id}$ for all $g \in G$. 
Here, the modularity appears to be stronger. 
The reason is that we made the assumption $\langle J_W \rangle \subset G$, 
i.e. $\text{deg}^\vee \in N$, 
so that we can combine $q_j = \theta_j(\text{deg}^\vee)$ with $\theta_j(n)$ 
and do a change of variable. 
We also assumed that $G \subset SL_W$, 
or equivalently $\text{deg} \in M$, 
then $\sum_j \theta_j(n) \in \mathbb Z$ for all $n \in G$. 
This enables us to reduce terms
$\prod_j e^{i 2 \pi \theta_j(n)}$ to $1$. 
\end{remark}

Now, we can prove that the elliptic genera of mirror Berglund-H\"ubsch models coincide up to a sign.

\begin{theorem}
$Ell(W/G, z, \tau) = (-1)^{\hat c} Ell(W^\vee/G^\vee, z, \tau)$. 
\end{theorem}

\begin{proof}
The Fock space $\text{Fock}_{M\oplus N}$ 
and the differential $D_{1, 1}$ 
are both symmetric with respect to the 
switching of $M$ and $N$. 
The discrepancy of the cocycle (\ref{coc}) and its counterpart with the role of $M$ and $N$ switched
can be resolved by multiplying $|m, n\rangle$ by $(-1)^{m\cdot n}$. 
However, 
to obtain the elliptic genus of the dual theory $W^\vee/G^\vee$ 
from the double-graded superdimension of $V_{1, 1}$, 
we need to consider a different bi-grading  
than the $(J[0], L[0])$ in (\ref{jfield}) and (\ref{lfield}). 
Instead, consider 
\begin{eqnarray}
J^*(z) = - J(z), \qquad 
L^*(z) = L(z) - \partial_z J(z), 
\end{eqnarray}
or 
\begin{eqnarray}
J^*[0] = - J[0], \qquad 
L^*[0] = L[0] + J[0].
\end{eqnarray}
Then by Theorem \ref{modularity}, we have
\begin{eqnarray}
Ell(W^\vee/G^\vee, y, q) 
& = & y^{- \frac{\hat c}{2}} 
\text{SuperTrace}_{V_{1, 1}} y^{J^*[0]} q^{L^*[0]} 
\nonumber \\
& = & y^{- \frac{\hat c}{2}} 
\text{SuperTrace}_{V_{1, 1}} (y^{-1} q)^{J[0]} q^{L[0]}
\nonumber \\
& = & y^{-\hat c} q^{\frac{\hat c}{2}} 
Ell(W/G, y^{-1} q, q).
\nonumber
\end{eqnarray}
It remains to use the following transformation property of $Ell$: 
\begin{eqnarray}
Ell(W/G, -z + \tau, \tau) & = & 
(-1)^{\hat c} e^{-i \pi \hat c (\tau - 2 z)}
Ell(W/G, z, \tau).
\nonumber
\end{eqnarray}
\end{proof}










\end{document}